\documentclass[a4paper]{article}
\usepackage[utf8]{inputenc}
\usepackage{geometry}
\usepackage{amsmath,amssymb,amsfonts}
\usepackage{colortbl}
\usepackage{pgf}
\usepackage{epsfig}
\usepackage{textcomp}
\usepackage{url}
\usepackage{esvect}
\usepackage{amsthm,amsmath,amsfonts}
\usepackage[nottoc, notlof, notlot]{tocbibind}
\newtheorem{theorem}{Theorem}[section]
\newtheorem{proposition}{Proposition}[section]
\newtheorem{corollary}{Corollary}[section]
\theoremstyle{remark}
\newtheorem{remark}{Remark}[section]
\title{Study of a family of higher order nonlocal degenerate parabolic equations: from the porous medium equation to the thin film equation}
\author{Rana Tarhini \footnote{Université Paris Est, Laboratoire d'Analyse et de Mathématiques Appliquées UMR 8050, 61 avenue du Général de Gaulle, 94010 Créteil, France.}}

\begin{document}
\maketitle

\begin{abstract}
In this paper, we study a nonlocal degenerate parabolic equation of order $\alpha +2$ for $\alpha \in (0,2)$. The equation is a generalization of the one arising in the modeling of hydraulic fractures studied by Imbert and Mellet in 2011. Using the same approach, we prove the existence of solutions for this equation for $0<\alpha<2$ and for nonnegative initial data satisfying appropriate assumptions. The main difference is the compactness results due to different Sobolev embeddings. Furthermore, for $\alpha>1$, we construct a nonnegative solution for nonnegative initial data under weaker assumptions.
\end{abstract}

\section{Introduction}

In this paper, we study the following problem
\begin{eqnarray}\label{eq:probl}
\begin{cases}
 \partial_t u + \partial_x (u^n \partial_x I(u)) = 0  & \quad \quad  \text{for } x \in \Omega ,\quad t > 0,  \\ 
 \partial_x u = 0 , u^n \partial_x I(u)= 0 & \quad \quad  \text{for } x \in \partial\Omega ,\quad t > 0, \\
u(0,x) = u_0(x) & \quad \quad  \text{for } x \in \Omega,
\end{cases}
\end{eqnarray}
where $\Omega = (a,b)$ is a bounded interval in $ \mathbb{R} $, $ n $ is a positive real number and $ I $ is a nonlocal elliptic negative operator of order $ \alpha $ defined as the $\alpha / 2$ power of the Laplace operator with Neumann boundary conditions $I = -(-\Delta)^\frac{\alpha}{2}$  where $\alpha \in (0,2)$; this operator will be defined below by using the spectral decomposition of the Laplacian.

The case $\alpha = 1$ was studied by Imbert and Mellet \cite{Imbert-mellet} who proved the existence of nonnegative solutions for nonnegative initial data with appropriate conditions. In this case, when $ n = 3 $ the equation designs the physical KGD model developed by Geertsma and de Klerk \cite{geertsma} and Khristianovich and Zheltov \cite{khristianovich}. It represents the influence of the pressure exerted by a viscous fluid on a fracture in an elastic medium subject only to plane strain. This equation is derived from the conservation of mass for the fluid inside the fracture, the Poiseuille law and an appropriate pressure law (see \cite[section 3]{Imbert-mellet} and \cite{IM2} for further details). In \cite{Imbert-mellet}, weak solutions are constructed by passing to the limit in a regularized problem. The necessary compactness estimates are obtained from appropriate energy estimates.

\noindent The equation under consideration
\begin{eqnarray} \label{equation}
u_t + \partial_x (u^n \partial_x I(u)) = 0
\end{eqnarray}
is a nonlocal degenerate parabolic equation of order $\alpha + 2$.

When $\alpha=2$, this equation coincides with the thin film equation (TFE for short) 
\begin{eqnarray}
u_t + \partial_x (u^n \partial_{xxx}^3 u) = 0.
\end{eqnarray} 
This is a fourth order nonlinear degenerate parabolic equation originally studied by Bernis and Friedman \cite{bernis}. This equation arises in many applications like spreading of a liquid film over a solid surface ($n=3$) and Hele-Shaw flows ($n=1$) (see \cite{thin1, thin2, thin3, thin4, thin5, thin6, thin7}). TFE is derived also from a conservation of mass, the Poiseuille law (derived from a lubrication approximation of the Navier-Stokes equations for thin film viscous flows) and various pressure laws. The parameter $n \in (0,3]$ models various boundary conditions at the liquid-solid interface. The case $n>3$ is mainly of mathematical interest \cite{comparaison}. In \cite{bernis} weak solutions $u$ are exhibited in a bounded interval under appropriate boundary conditions. In addition, they proved that $u$ is nonnegative if $u_0$ is also so, and that the support of the solution $u(t,.)$ increases with $t$ if $u_0$ is nonnegative and $n \geq 4$.

For $\alpha=0$, the porous medium equation (PME for short) is recovered
\begin{eqnarray}
u_t - \partial_x (u^n \partial_x u) = 0.
\end{eqnarray}
This is a nonlinear degenerate parabolic equation. The simple PME model describes the modeling of the motion of a gas flow through a porous medium \cite{PME}. In this case, the PME is derived from mass balance, Darcy's law which describes the dynamics of flows through porous media, and a state equation for the pressure \cite{PME}. PME also arises in heat transfer \cite{heattransfer} and groundwater flow \cite{groundwater} and was originally proposed by Boussinesq. It took many years to prove that PME is well posed and the famous source type solutions were found by Zel'dovich, Kompanyeets and Barenblatt \cite{PME}. The questions of existence, uniqueness, stability, smoothness of solutions together with dynamical properties and asymptotic behavior are well represented in \cite{PME} where two main problems are studied. First, the domain space is $\mathbb{R}^d$ and the initial condition $u_0$ has a compact support so the solution $u(t,x)$ vanishes for all positive times $t > 0$ outside a compact set that changes with time. Secondly, if the initial data has a hole in the support then the solution has a possibly smaller hole for $t > 0$.

Note that TFE can be seen as a fourth order version of the classical PME \cite{comparaison}. Furthermore, both equations are parabolic in divergence form. In both cases, there are compactly supported source type solutions ($n>1$ for PME \cite{PME} and $0<n<3$ for TFE \cite{sourcetypeforTFE}) \cite{similarities}. The most famous common properties are finite speed of propagation and the waiting time phenomenon.  Similar properties are expected in our case. Self-similar solutions are constructed in \cite{IM2} but other properties are still not proved. One striking difference between TFE and PME is the lack of a maximum principle for TFE \cite{similarities}.

The case $\alpha \in (-2,0)$ corresponds to the fractional porous medium equation studied  in \cite{FPME}. Explicit self-similar solutions are exhibited and, under appropriate conditions, weak solutions are constructed.

\bigskip

In this paper, we will generalize the result of \cite{Imbert-mellet} to the cases $ 0 < \alpha \leqslant 1 $ and $ 1 < \alpha < 2 $. We prove a result of existence with the same approach as that in the case $\alpha = 1$ but by modifying the compactness results. Consequently all cases $\alpha \in [0,2]$ are now covered.

In the case $ \alpha > 1 $ we get the local uniform convergence of approximate solutions due to the following embedding in dimension $1$
\begin{eqnarray*}
H^\frac{\alpha}{2}(\Omega) \hookrightarrow C^{0,\frac{\alpha-1}{2}}(\Omega).
\end{eqnarray*}
This convergence allows one to pass to the limit in the nonlinear term and then allows us to construct nonnegative solutions for nonnegative initial data merely in $H^\frac{\alpha}{2}(\Omega)$.

In the case $\alpha < 1$ because of the following embedding
\begin{eqnarray*}
H^\frac{\alpha}{2}(\Omega) \hookrightarrow L^p(\Omega) \text{ for all } p < \frac{2}{1-\alpha},
\end{eqnarray*}
we can get a compactness result in $L^p(\Omega)$ only for $p < \frac{2}{1-\alpha}$ and not for all $p < \infty$ as in the case $\alpha = 1$. Neverthless, we recover a compactness result for the term $I(u)$ which allows us to pass to the limit and conclude.

In both cases, we prove that the solution is strictly positive under a condition on $n$.

\subsection*{Integral inequalities}

Assume that $\Omega = \mathbb{R}$, if $u$ is a solution of \eqref{equation} then it satisfies the energy inequality
\begin{eqnarray*}
- \int_\Omega u(t) I(u(t)) dx + 2 \int_0^T \int_\Omega u^n \partial_x I(u)^2 dx dt \leqslant - \int_\Omega u_0 I(u_0) dx.
\end{eqnarray*}
Observe that $-\int u I(u)$ is the homogeneous $H^\frac{\alpha}{2}$ norm. 
Let $G$ be a nonnegative function such that $G''(s) = \frac{1}{s^n}$. Then the positive solution satisfies
\begin{eqnarray*}
\int_\Omega G(u(t)) dx - \int_0^T \int_\Omega \partial_x u \partial_x I(u) dx dt \leqslant \int_\Omega G(u_0) dx.
\end{eqnarray*}
Note that $-\int \partial_x u  \partial_x I(u)$ is the homogeneous $H_N^{\frac{\alpha}{2}+1}$ norm (it is in fact a Neumann-Sobolev space, see below). We see that the energy inequality controls the $L^\infty(0,T;H^\frac{\alpha}{2}(\Omega))$ norm of the solution. For the function $ G $ mentioned above, we can take
 \begin{eqnarray} \label{G}
 G(s) = \int_1^s \int_1^r \frac{1}{t^n} dt dr
 \end{eqnarray}
 so that $ G $ is a nonnegative convex function satisfying $G(1) = G'(1) = 0$, $ G(s) = \infty $ for all $ s < 0 $ and for $s > 0$, we have
 \begin{eqnarray*}
 G(s) = \begin{cases} s \ln s - s + 1 \quad & \text{ when } n = 1\\
 - \frac{s^{2-n}}{(2-n)(n-1)} + \frac{s}{n-1} + \frac{1}{2-n} \quad & \text{ when } 1 < n < 2 \\
 \ln \frac{1}{s} + s - 1 \quad & \text{ when } n = 2 \\
 \frac{1}{(n-2)(n-1)}\frac{1}{s^{n-2}} + \frac{s}{n-1} - \frac{1}{n-2} \quad & \text{ when } n > 2. 
 \end{cases}
 \end{eqnarray*}

\subsection*{Main results}
In this work, we prove three main results. We first prove the existence of nonnegative weak solutions for the problem with $ 0 < \alpha \leq 1$ for nonnegative initial data with apropriate conditions. Secondly, for $\alpha>1$, we construct nonnegative solutions for nonnegative initial data in $H^\frac{\alpha}{2}(\Omega)$. Finally, we prove the strict positivity of solutions for large $n's$.

\begin{theorem}[Existence of solutions for $0< \alpha \leqslant 1$] \label{th:th1}

Let $ n \geqslant 1 $ and $\alpha \in (0,1]$.
For any nonnegative initial condition $ u_0 \in H^{\frac{\alpha}{2}}(\Omega) $ such that 
\begin{eqnarray} \label{entropy condition}
\int_\Omega G(u_0) dx < \infty
\end{eqnarray}
where G is a nonnegative function such that $ G''(s)= \frac{1}{s^n} $, there exists a nonnegative function
\begin{eqnarray*}
u \in L^\infty(0,T ; H^{\frac{\alpha}{2}}(\Omega)) \cap L^2(0,T ; H_N^{\frac{\alpha}{2} + 1}(\Omega))
\end{eqnarray*} 
which satisfies on $Q=(0,T)\times \Omega$
\begin{eqnarray}\label{eq:formu}
\iint_Q u \partial_t \varphi dt dx - \iint_Q n u^{n-1} \partial_x u I(u) \partial_x \varphi dx dt - \iint_Q u^n I(u) \partial^2_{xx} \varphi dx dt = -\int_\Omega u_0 \varphi(0,.) dx
\end{eqnarray}
 for all $ \varphi \in \mathcal{D}([0,T)\times \bar{\Omega}) $ satisfying $ \partial_x \varphi = 0 $ on $(0,T) \times \partial\Omega $.

Furthermore u satisfies for almost every $ t \in (0,T) $
\begin{eqnarray}\label{masscons}
\int_\Omega {u(t,x)dx} = \int_\Omega {u_0(x) dx}
\end{eqnarray}
and
\begin{eqnarray}\label{eq:ineq}
  \| {u(t,.)}\|_{\overset{.}{H}^\frac{\alpha}{2}(\Omega)}^2 + 2\int_0^T \int_\Omega g^2 dx ds \leq \| u_0 \|_{\overset{.}{H}^\frac{\alpha}{2}(\Omega)}^2 
 \end{eqnarray}
 where the function $ g \in L^2(Q) $ satisfies $ g = \partial_x (u^\frac{n}{2} I(u)) - \frac{n}{2} u^\frac{n-2}{2} \partial_x u I(u) $ in $ \mathcal{D}'(\Omega) $, and
 \begin{eqnarray}\label{eq:entropy} 
\int_\Omega G(u(t,x)) dx + \int_0^t{\| u\|_{\overset{.}{H}_N^{\frac{\alpha}{2}+1}(\Omega)}^2 ds \leq \int_\Omega G(u_0) dx}.
 \end{eqnarray} 
 \end{theorem}
 
 \begin{remark}
  The weak formulation \eqref{eq:formu} comes after two integrations by parts of the equation \eqref{equation}. We recall that the function $G:\mathbb{R}_+ \rightarrow \mathbb{R}_+$ is given by \eqref{G}. Note that the space $H_N^s(\Omega)$ is defined via spectral decomposition of $-(-\Delta)^\frac{s}{2}$ (see below).
 \end{remark}
 
 \begin{theorem}[Existence of solutions for $1 < \alpha < 2$] \label{th:th2}
Let $n\geqslant1$ and $\alpha>1$.
For any nonnegative initial condition $u_0 \in H^{\frac{\alpha}{2}}(\Omega)$, there exists a nonnegative function 
\begin{eqnarray*}
u \in C^{\frac{\alpha-1}{2(\alpha+2)}, \frac{\alpha-1}{2}}_{t,x}(Q) \
\end{eqnarray*}
such that
\begin{eqnarray}
\partial_x I(u) \in L_{loc}^2(Q_+)
\end{eqnarray}
and that satisfies
\begin{eqnarray} \label{limitP}
\iint_Q u \partial_t \varphi dt dx + \iint_{Q_+} u^n \partial_x I(u) \partial_x \varphi dx dt = -\int_\Omega u_0 \varphi(0,.) dx
\end{eqnarray}
where $Q_+ = \{u>0\} \cap Q$, for all $ \varphi \in \mathcal{D}([0,T)\times \bar{\Omega}) $ satisfying $\partial_x \varphi = 0 $ on $(0,T) \times \partial\Omega $. 
Furthermore, $u$ satisfies conservation of mass.
\end{theorem}

\begin{theorem}[Strictly positive solutions] \label{th:th3}
Assume $ 0 < \alpha < 2$ and $n > 2+ \frac{2}{\alpha+1}$. There exists a set $ P \subset (0,T) $ such that $ \mid (0,T)\setminus P \mid = 0 $ and the solution $ u $ constructed as in Theorem \ref{th:th1} satisfies $ u(t,.) \in C^{0,\beta}(\Omega) $ for all $ t \in P $ and for all $ \beta < min\{1,\frac{\alpha + 1}{2}\} $ and $ u(t,.) $ is strictly positive in $ \Omega $. Furthermore, u is a solution of
\begin{eqnarray*}
u_t + \partial_x J = 0 \quad \quad \quad  \text{in } \mathcal{D}'(\Omega)
\end{eqnarray*}
where
\begin{eqnarray*}
J(t,.) = u^n \partial_x I(u) \in L^1(\Omega) \quad \quad \quad \text{for all } t\in P.
\end{eqnarray*}
\end{theorem}

\subsection*{Organization of the paper}

The paper is organized as follows: in Section 2, we define the nonlocal operator $I $ by using the spectral decomposition of the Laplacian and we write an integral representation for it. Then we prove two important Propositions used in the proofs. In Section 3, we study a regularized problem before proving our Theorems in Section 4.  

\subsection*{Notation}
In this work, we denote $\Omega = (0,1)$ and $ Q = (0,T) \times \Omega $. The space $ H^s_N(\Omega) $ is the functional space defined in \cite[Section 3.1]{Imbert-mellet} by
\begin{eqnarray*}
H^s_N(\Omega) = \left\{ u=\sum_{k=0}^\infty c_k \varphi_k ; \sum_{k=0}^\infty c_k^2(1+\lambda_k^s) < +\infty \right\}
\end{eqnarray*}
where $ \{ \lambda_k , \varphi_k \}_{k \geq 0} $ are the eigenvalues and corresponding eigenvectors of the Laplacian operator in $\Omega $ with Neumann boundary conditions on $ \partial\Omega $ with the norm 
\begin{eqnarray}
\| u \|_{H^s_N(\Omega)}^2 = \sum_{k=0}^\infty c_k^2(1+\lambda_k^s),
\end{eqnarray}
equivalently to
\begin{eqnarray}
\| u \|_{H^s_N(\Omega)}^2 = \left(\int_\Omega u dx\right)^2 + \| u \|_{\overset{.}{H}^s_N(\Omega)}^2
\end{eqnarray}
where the homogeneous norm is given by
\begin{eqnarray}
\| u \|_{\overset{.}{H}_N^s(\Omega)}^2 = \sum_{k=0}^\infty c_k^2 \lambda_k^s.
\end{eqnarray}
Note that $H^s_N(\Omega) = H^s(\Omega)$ for all $0\leqslant s < \frac{3}{2}$ (see \cite{sobolevspaces}) with equivalent norms. Indeed,
\begin{eqnarray*}
\|u\|_{H^s(\Omega)}^2 = \|u\|_{L^2(\Omega)}^2 + \|u\|_{\overset{.}{H}^s(\Omega)}^2 
\end{eqnarray*}
and since we are in dimension 1 we have for these values of $s$
\begin{eqnarray*}
\|u\|_{\overset{.}{H}^s_N(\Omega)} = \|u\|_{\overset{.}{H}^s(\Omega)}
\end{eqnarray*}
Note also that we have 
\begin{eqnarray*}
& \int_\Omega u dx \leqslant C(\Omega) \|u\|_2 \quad (\text{H$\overset{..}{o}$lder inequality}),\\
& \|u\|_2^2 \leqslant C(\Omega) \|(-\Delta)^\frac{s}{2}u\|_2^2 \leqslant c \|u\|_{\overset{.}{H}^s_N(\Omega)}^2 (\text{fractional Poincaré's inequality}).
\end{eqnarray*}
Finally, as usual $ s_+ = $ max$\{0,s\} $. 

\section{Preliminaries}
\subsection{Operator I}
\textit{Spectral definition.}\quad We define the operator I by
\begin{eqnarray*}
I : \sum_{k = 0}^\infty c_k \varphi_k \longrightarrow -\sum_{k = 0}^\infty c_k \lambda_k^\frac{\alpha}{2} \varphi_k \quad \quad \text{ which maps } H_N^\alpha(\Omega) \text{ onto } L^2(\Omega) 
\end{eqnarray*}
where $ \{ \lambda_k , \varphi_k \}_{k \geq 0} $ are the eigenvalues and corresponding eigenvectors of the Laplacian operator in $\Omega $ with Neumann boundary conditions on $ \partial\Omega $:
\begin{eqnarray*}
\begin{cases}
-\Delta \varphi_k = \lambda_k \varphi_k  \text{ in } \Omega, \\
\partial_\nu \varphi_k = 0  \text{ on } \partial\Omega, \\
\int_\Omega \varphi_k^2 dx = 1.
\end{cases}
\end{eqnarray*}

\textit{Integral representation.}\quad The operator $I$ can also be represented as a singular integral operator. We will prove the following.
\begin{proposition}
Consider a smooth function $u:\Omega \rightarrow \mathbb{R}$. Then for all $x\in\Omega$,
\begin{eqnarray*}
I(u)(x) = \int_\Omega (u(y) - u(x)) K(x,y) dy
\end{eqnarray*}
where $K(x,y)$ is defined as follows. For all $x, y \in \Omega$ 
\begin{eqnarray*}
K(x,y) = c_\alpha \sum_{k\in\mathbb{Z}} \left( \frac{1}{\mid x - y -2k \mid^{1+\alpha}} + \frac{1}{\mid x + y -2k \mid^{1+\alpha}} \right)
\end{eqnarray*}
where $c_\alpha$ is a constant depending only on $\alpha$.
\end{proposition}

\begin{proof}
Let's replace $\Omega$ by $(-1,1)$ and $u$ by its even extension to $(-1,1)$. Then let's extend $u$ periodically to $\mathbb{R}$ and let $\bar{u}$ be this extension.
For $x\in \Omega$,
\begin{align*}
I(u)(x) & = - (-\Delta)^\frac{\alpha}{2} \bar{u}(x)
= c_\alpha \int_\mathbb{R} (\bar{u}(y)-\bar{u}(x)) \frac{dy}{\mid y-x \mid^{1+\alpha}} \\
 & = c_\alpha \sum_{k\in\mathbb{Z}} \int_{-1+2k}^{1+2k} (\bar{u}(y)-u(x)) \frac{dy}{\mid y-x \mid^{1+\alpha}} \\ & = c_\alpha \int_{-1}^1 (\bar{u}(y)-u(x)) \left( \sum_{k\in\mathbb{Z}}\frac{1}{\mid y+2k-x \mid^{1+\alpha}} \right) dy \quad & \text{ because } \bar{u} \text{ is 2-periodic} \\ 
& = c_\alpha \int_0^1 (u(y) - u(x))\sum_{k\in\mathbb{Z}} \left( \frac{1}{\mid x - y -2k \mid^{1+\alpha}} + \frac{1}{\mid x + y -2k \mid^{1+\alpha}} \right) \quad & \text{ because } \bar{u} \text{ is even.}
\end{align*}
\end{proof}
 Now we can easily conclude the following Corollary.
\begin{corollary} \label{corollary}
Consider two smooth functions $u,\varphi:\Omega \rightarrow \mathbb{R}$. Then 
\begin{eqnarray}
\int_\Omega I(u)(x) \varphi(x) dx = \int_\Omega u(x) I(\varphi)(x) dx
\end{eqnarray} 
\end{corollary}

\subsection{Important identities}
As \cite[Section 3]{Imbert-mellet}, the semi-norms $\| . \|_{\overset{.}{H}^\frac{\alpha}{2}(\Omega)} $, $\| . \|_{\overset{.}{H}^\alpha_N(\Omega)} $, $\| . \|_{\overset{.}{H}^{\frac{\alpha}{2}+1}_N(\Omega)} $ and $\| . \|_{\overset{.}{H}^{\alpha + 1}_N(\Omega)} $ are related to the operator $ I $ by important and very useful equalities. 

\begin{proposition} \label{identities}
\begin{enumerate}
\item For all $u \in H^\frac{\alpha}{2}(\Omega),$  we have $- \langle I(u),u \rangle = \| u \|^2_{\overset{.}{H}^\frac{\alpha}{2}(\Omega)}.$
\item For all $u \in H_N^\alpha(\Omega),$  we have $\| u \|_{\overset{.}{H}^\alpha_N(\Omega)}^2 = \int_\Omega I(u)^2 dx.$
\item For all $u \in H^{\frac{\alpha}{2}+1}_N(\Omega),$  we have $\| u \|_{\overset{.}{H}^{\frac{\alpha}{2}+1}_N(\Omega)}^2 = - \int_\Omega I(u)_x u_x dx.$
\item For all $u \in H^{\alpha + 1}_N(\Omega),$  we have $\| u \|_{\overset{.}{H}^{\alpha + 1}_N(\Omega)}^2 = \int_\Omega I(u)_x^2 dx.$
\end{enumerate}
\end{proposition}

\begin{proof}
Note that if $ u \in H^\frac{\alpha}{2}(\Omega) $ then $ I(u) \in H^{-\frac{\alpha}{2}}(\Omega) $ and
\begin{eqnarray*}
\langle I(u),v \rangle_{H^{-\frac{\alpha}{2}}(\Omega),H^\frac{\alpha}{2}(\Omega)} = - \sum_{k = 0}^\infty c_k \lambda_k^\frac{\alpha}{2} d_k
\end{eqnarray*}
where $ v = \sum_{k = 0}^\infty d_k \varphi_k \in H^\frac{\alpha}{2}(\Omega) $ and $ u = \sum_{k = 0}^\infty c_k \varphi_k $, so
\begin{eqnarray*}
 - \int u I(u) = \sum_{k = 0}^\infty c_k^2 \lambda_k^\frac{\alpha}{2} = \| u \|^2_{\overset{.}{H}^\frac{\alpha}{2}(\Omega)}.
\end{eqnarray*}

The second equality is actually very easy to prove since $ I(u) = -\sum_{k = 0}^\infty c_k \varphi_k^\frac{\alpha}{2} $. Indeed,
\begin{eqnarray*}
\int_\Omega I(u)^2 dx = \sum_{k = 0}^\infty c_k^2 \varphi_k^\alpha = \| u \|^2_{\overset{.}{H}_N^\alpha(\Omega)}.
\end{eqnarray*}

In order to prove the other equalities, we note that $ (\partial_x\varphi_k)_k $ form an orthogonal basis of $ L^2(\Omega) $.
We write
\begin{align*}
 & u_x = \sum_{k = 0}^\infty c_k \partial_x\varphi_k  \text{ in } L^2(\Omega). \\
\text{and } & \partial_x I(u) = -\sum_{k = 1}^\infty c_k \lambda_k^\frac{\alpha}{2} \partial_x\varphi_k \text{ in } L^2(\Omega)
\end{align*}
so
\begin{align*}
- \int_\Omega I(u)_x u_x dx & =  \sum_{k = 0}^\infty c_k^2 \lambda_k^\frac{\alpha}{2} \int_\Omega \partial_x\varphi_k^2 dx = \sum_{k = 0}^\infty c_k^2 \lambda_k^\frac{\alpha}{2} \int_\Omega \varphi_k (- \partial_{xx}\varphi_k) dx \\
& = \sum_{k = 0}^\infty c_k^2 \lambda_k^\frac{\alpha}{2} \int_\Omega \lambda_k \varphi_k^2 dx = \sum_{k = 0}^\infty c_k^2 \lambda_k^{\frac{\alpha}{2}+1} = \| u \|_{\overset{.}{H}^{\frac{\alpha}{2}+1}_N(\Omega)}^2.
\end{align*}
For the last equality,
\begin{align*}
\int_\Omega I(u)_x^2 dx & = \sum_{k = 1}^\infty c_k^2 \lambda_k^\alpha \int_\Omega \partial_x \varphi_k^2 dx = \sum_{k = 0}^\infty c_k^2 \lambda_k^{\alpha + 1}  = \| u \|_{\overset{.}{H}^{\alpha + 1}_N(\Omega)}^2.
\end{align*}
\end{proof}

\subsection{The problem $-I(u)= g $}

We consider the following problem
\begin{eqnarray}\label{eq:statprob}
\begin{cases}
\text{For a given } g \in L^2(\Omega), \text{find } u \in H^\alpha_N(\Omega) \text{such that } \\
 \quad \quad \quad \quad \quad \quad -I(u) = g.
\end{cases}
\end{eqnarray}

Since $ \int_\Omega I(u)dx = 0 $ for all $ u \in H_N^\alpha(\Omega) $, we must assume that $ \int_\Omega g(x) dx = 0 $ otherwise \eqref{eq:statprob} has no solution.

\begin{proposition} \label{prop:prob}
For all $ g \in L^2(\Omega) $ such that $ \int_\Omega g dx = 0 $, there exists a unique function $ u \in H^\alpha_N(\Omega) $ such that 
\begin{eqnarray*}
-I(u) = g \quad  \text{in } L^2(\Omega)  \text{ and } \int_\Omega u dx = 0.
\end{eqnarray*}
Furthermore if $ g \in H^1(\Omega) $, then $ u \in H^{\alpha + 1}_N(\Omega) $.
\end{proposition}

\begin{proof}
Let $ g \in L^2(\Omega) $.
For $ g = \sum_{k=1}^\infty d_k \varphi_k $ with $ \sum_{k=1}^\infty d_k^2 < \infty $, we consider
\begin{eqnarray*}
u = I^{-1} (g) =  \sum_{k=1}^\infty \frac{d_k}{\lambda_k^\frac{\alpha}{2}} \varphi_k \in H^\alpha_N(\Omega) \text{ and verify } \int_\Omega u dx = 0.
\end{eqnarray*}
Since $ (\varphi_k)_k $ form an orthogonal basis of $ L^2(\Omega) $, the solution is the unique satisfying $ \int_\Omega u dx = 0 $.
It is clear that every further regularity on $ g $ will imply a further regularity on $ u $ shifted by an $\alpha$.
\end{proof}
We thus conclude the following Corollary which will be used to prove the existence of solutions for the stationary problem.
\begin{corollary} \label{cor:bijective}
For all $ g \in L^2(\Omega) $, there exists a unique function $ v \in H^\alpha_N(\Omega) $ such that 
\begin{eqnarray} \label{eq:v}
-I(v) + \int_\Omega v dx = g.
\end{eqnarray}
Furthermore if $ g \in H^1(\Omega) $, then $ u \in H^{\alpha + 1}_N(\Omega) $ and the map $g \rightarrow u$ is bijective.
\end{corollary}
\begin{proof}
Let $ m = \int_\Omega g dx $ and $ g' = g - m $. Then $ g' \in L^2(\Omega) $(since $\Omega$ is bounded) and $ \int_\Omega g' dx = 0 $. From Proposition \ref{prop:prob}, there exists a function $ u \in H^\alpha_N(\Omega) $ such that 
\begin{eqnarray*}
-I(u) = g' \quad \text{ and } \int_\Omega u dx = 0.
\end{eqnarray*}
Let $ v = u + m $. Then $ \int_\Omega v dx = m $ and 
\begin{eqnarray*}
- I(v) = - I(u) = g' = g - m = g - \int_\Omega v dx.
\end{eqnarray*}
For the uniqueness, consider two solutions $ v_1 $ and $ v_2$ then
\begin{eqnarray*}
\int_\Omega v_1 dx = \int_\Omega v_2 dx = \int_\Omega g dx 
\end{eqnarray*}
and $ w = v_1 - v_2 $ satisfies $ -I(w) = 0 $. Hence, $ w = 0 $ from the uniqueness given by Proposition~\ref{prop:prob}.
\end{proof}

\section{Regularized problem}

We consider the following regularized problem
\begin{eqnarray} \label{pro:reg}
\begin{cases}
 \partial_t u + \partial_x (f_\epsilon(u) \partial_x I(u)) = 0  &  \text{ for } x \in \Omega ,\quad t > 0, \\ 
 \partial_x u = 0 , f_\epsilon(u) \partial_x I(u) = 0 & \text{ for } x \in \partial\Omega ,\quad t > 0, \\
u(x,0) = u_0(x) &   \text{ for } x \in \Omega,
\end{cases}
\end{eqnarray}
where $ f_\epsilon(s) = s_+^n + \epsilon , \quad \epsilon > 0 $ and $0 < \alpha < 2$.

To prove Theorem \ref{th:th1} and \ref{th:th2}, we need to prove the existence of a solution for the regularized problem. Let us pass with the following \textbf{stationary problem}
\begin{eqnarray} \label{pro:statio}
\text{For } \tau > 0 , g \in H^\frac{\alpha}{2}(\Omega), \text{ find } u \in H_N^{\alpha + 1}(\Omega) \text{ s.t.} 
\begin{cases}
u + \tau \partial_x (f_\epsilon(u) \partial_x I(u)) = g &  \text{ in } \Omega, \\
\partial_x u = 0 , \partial_x I(u) = 0 &  \text{ on } \partial\Omega.
\end{cases} 
\end{eqnarray}
Once we get a solution for \eqref{pro:statio}, we can prove the existence of a solution for \eqref{pro:reg}.

\subsection{Stationary problem}

\begin{proposition} \label{prop:exis}
For all $ g \in H^\frac{\alpha}{2}(\Omega)  $, there exists $ u \in H^{\alpha + 1}_N(\Omega) $ such that for all $ \varphi \in H^1(\Omega) $ we have
\begin{eqnarray} \label{eq: pro}
\int_\Omega u \varphi dx - \tau~\int_\Omega f_\epsilon(u) \partial_x I(u) \partial_x \varphi dx = \int_\Omega g \varphi dx.
\end{eqnarray}
Furthermore, $u$ verifies 
\begin{eqnarray} \label{eq:mass}
\int_\Omega u(x) dx = \int_\Omega g(x) dx 
\end{eqnarray}
and
\begin{eqnarray} \label{eq:ine}
\| u \|^2_{\overset{.}{H}^\frac{\alpha}{2}(\Omega)} + 2 \tau \int_\Omega f_\epsilon(u) \partial_x I(u)^2 dx \leqslant \| g \|^2_{\overset{.}{H}^\frac{\alpha}{2}(\Omega)}.
\end{eqnarray}
If $ \int_\Omega G_\epsilon(g) dx < \infty $ where $ G_\epsilon $ is a nonnegative function such that $ G''_\epsilon(s) = \frac{1}{f_\epsilon(s)} $, then
\begin{eqnarray} \label{eq:entro}
\int_\Omega G_\epsilon(u) dx + \tau \| u \|^2_{\overset{.}{H}^{\frac{\alpha}{2}+1}_N(\Omega)} \leqslant \int_\Omega G_\epsilon(g) dx.
\end{eqnarray}
\end{proposition}
\begin{remark}
 Note that we can consider $G_\epsilon(s) = \int_1^s \int_1^t G_\epsilon''(r) dr dt$, so $G_\epsilon$ is a non-
 negative convex function for all $\epsilon > 0$ satisfying $G_\epsilon(1) = G_\epsilon'(1) = 0$.
\end{remark}
\begin{proof}
Thanks to Corollary \ref{cor:bijective}, we can recover all test functions from $H^1(\Omega)$ by considering 
\begin{eqnarray*}
\varphi = - I(v) + \int_\Omega v dx 
\end{eqnarray*}
for some function $v \in H^{\alpha+1}_N (\Omega)$. So equation \eqref{eq: pro} becomes 
\begin{eqnarray} \label{formulation}
- \int_\Omega u I(v) dx & + \left( \displaystyle\int_\Omega u dx\right) \left( \displaystyle\int_\Omega v dx\right) + \tau \displaystyle\int_\Omega f_\epsilon(u) \partial_x I(u) \partial_x I(v) dx  \nonumber \\
 & = - \displaystyle\int_\Omega g I(v) dx + \left( \displaystyle\int_\Omega g dx\right)\left( \displaystyle\int_\Omega v dx\right).
\end{eqnarray}
Now, we consider the nonlinear operator A defined by
\begin{eqnarray*}
A(u)(v) = - \displaystyle\int_\Omega u I(v) dx & + \left( \displaystyle\int_\Omega u dx\right) \left( \displaystyle\int_\Omega v dx\right) + \tau \displaystyle\int_\Omega f_\epsilon(u) \partial_x I(u) \partial_x I(v) dx \text{ for } u,v \in H^{\alpha+1}_N (\Omega).
\end{eqnarray*}
We prove that this is a continuous, coercive and pseudo-monotone operator. Note that the functional $T_g$ defined by
\begin{eqnarray*}
T_g(v) = - \int_\Omega g I(v) dx + \left(\int_\Omega g dx\right)\left(\int_\Omega v dx\right) \text{ for } v \in H^{\alpha+1}_N (\Omega).
\end{eqnarray*}
is a linear form on $ H^{\alpha + 1}_N(\Omega) $. So our problem reduces to the following
\begin{eqnarray} \label{pseudo}
\begin{cases}
\text{Let } V = H_N^{\alpha+1}(\Omega).\\ A : V \rightarrow V^* \text{ coercive, continuous and pseudo-monotone}. \\ T_g \in V^*. \\ \text{Find } u \in H_N^{\alpha+1}(\Omega) \text{ such that } A(u)=T_g \text{ in } V^*.
\end{cases}
\end{eqnarray}
The theory of pseudo-monotone operators \cite{pseudo} implies the existence of a solution for \eqref{pseudo} so there exists $ u \in H^{\alpha +1}_N(\Omega) $ such that
\begin{eqnarray*}
A(u)(v) = T_g(v) \quad \text{ for all } v \in H^{\alpha + 1}_N(\Omega).
\end{eqnarray*}
It remains to prove that $ A $ is a continuous, coercive and pseudo-monotone operator on $ H^{\alpha + 1}_N(\Omega) $. The reader can find the proof in \cite[Appendix A]{Imbert-mellet} for $V = H_N^2(\Omega)$ but this proof can be easily adapted for our case $V = H_N^{\alpha + 1}(\Omega)$.

By using Corollary \ref{cor:bijective} we deduce that $u$ satisfies \eqref{eq: pro} for all $\varphi \in H^1(\Omega).$\\
For the properties of $ u $, first by taking $ \varphi = 1 $ as a test function in \eqref{eq: pro} we obtain mass conservation~\eqref{eq:mass}.
Secondly, take $ v = u - \int_\Omega u dx $ in \eqref{formulation}, by using Proposition \ref{identities} we have
\begin{align*}
\| u \|^2_{\overset{.}{H}^\frac{\alpha}{2}(\Omega)} + & \tau \int_\Omega f_\epsilon(u) \partial_x I(u)^2  = - \int_\Omega g I(u) dx  \\
& \leqslant \| g \|_{\overset{.}{H}^\frac{\alpha}{2}(\Omega)} \| u \|_{\overset{.}{H}^\frac{\alpha}{2}(\Omega)} \\
& \leqslant \frac{1}{2} \| g \|^2_{\overset{.}{H}^\frac{\alpha}{2}(\Omega)} + \frac{1}{2} \| u \|^2_{\overset{.}{H}^\frac{\alpha}{2}(\Omega)}
\end{align*}
which \eqref{eq:ine}(Note that the high regularity of $g$ is solely used in this inequality, otherwise $g \in L^2(\Omega)$ is sufficient to prove the existence above).
Finally, note that $ G_\epsilon'$ is smooth with $G_\epsilon'$ and $G_\epsilon'' $ are bounded, and $ \Omega $ is bounded so we can take $ \varphi = G_\epsilon'(u) \in H^1(\Omega) $ as a test function in \eqref{eq: pro},
\begin{eqnarray*}
\int_\Omega u G_\epsilon'(u) dx - \tau \int_\Omega f_\epsilon(u) \partial_x I(u) \partial_x u G_\epsilon''(u) dx = \int_\Omega g G_\epsilon'(u) dx.
\end{eqnarray*}
So by using Proposition \ref{identities} and the fact that $ G_\epsilon''(s)=\frac{1}{f_\epsilon(s)}$ we get
\begin{eqnarray*}
\tau \| u \|^2_{\overset{.}{H}_N^{\frac{\alpha}{2}+1}(\Omega)}  = \int_\Omega G_\epsilon'(u) (g-u) dx \leqslant \int_\Omega (G_\epsilon(g) - G_\epsilon(u) ) dx
\end{eqnarray*}
because $G_\epsilon$ is convex and we deduce \eqref{eq:entro}.
\end{proof}

\subsection{Implicit Euler scheme}

We construct a piecewise constant function
\begin{eqnarray*}
u^\tau(t,x) = u^k(x) \text{ for } t \in [k \tau , (k+1) \tau) , k \in \{ 0,...,N-1 \} 
\end{eqnarray*}
where $ \tau = \frac{T}{N} $ and $ (u^k)_{k \in \{ 0,...,N-1 \} } $ is such that
\begin{eqnarray*}
u^{k + 1} + \tau~\partial_x(f_\epsilon(u^{k + 1}) \partial_x I(u^{k + 1})) = u^k.
\end{eqnarray*}
The existence of the $ u^k $ follows from Proposition \ref{prop:exis} by induction on $ k $ with $u^0 = u_0$. We deduce the following

\begin{corollary} \label{cor:existence}
For any $ N > 0 $ and $ u_0^\epsilon \in H^\frac{\alpha}{2}(\Omega) $, there exists a function $ u^\tau \in L^\infty(0,T;H^\frac{\alpha}{2}(\Omega)) $ such that 
\begin{enumerate}
\item $ t \longrightarrow u^\tau(t,x) $ is constant on $ [k \tau , (k+1) \tau) , k \in \{ 0,...,N-1 \} $ and $ \tau = \frac{T}{N} $.
\item $ u^\tau = u_0 $  on $ [0,\tau) \times \Omega $.
\item For all $ t \in (0,T) $,
\begin{eqnarray} \label{eq:mass u}
\int_\Omega u^\tau (t,x) dx = \int_\Omega u_0(x) dx.
\end{eqnarray}
\item For all $ \varphi \in C_c^1(0,T; H^1(\Omega)) $,
\begin{eqnarray} \label{eq:equ}
\iint_{Q_{\tau , T}} \frac{u^\tau - S_\tau u^\tau}{\tau} \varphi dx dt = \iint_{Q_{\tau , T}} f_\epsilon(u^\tau) \partial_x I(u^\tau) \partial_x \varphi dx dt
\end{eqnarray}
where $ S_\tau u^\tau (t,x) = u^\tau (t-\tau , x) $ and $ Q_{\tau , T} = (\tau , T) \times \Omega $.
\item For all $ t \in (0,T) $,
\begin{eqnarray} \label{eq:ineq u}
\| u^\tau(t,.) \|^2_{\overset{.}{H}^\frac{\alpha}{2}(\Omega)} + 2 \int_0^T \int_\Omega f_\epsilon(u^\tau) \partial_x I(u^\tau)^2 dx dt \leqslant \| u_0 \|^2_{\overset{.}{H}^\frac{\alpha}{2}(\Omega)}.
\end{eqnarray}
\item If $ \int_\Omega G_\epsilon(u_0) dx < \infty $, then for all $ t \in (0,T) $
\begin{eqnarray} \label{eq: entro u}
\int_\Omega G_\epsilon(u^\tau(t,x)) dx + \int_0^t \| u^\tau(s,.) \|^2_{\overset{.}{H}_N^{\frac{\alpha}{2}+1}(\Omega)} ds \leqslant \int_\Omega G_\epsilon(u_0) dx.
\end{eqnarray}
\end{enumerate}
\end{corollary}

\subsection{Existence of solution for the regularized problem}

Now we are able to prove the existence of a solution for the regularized problem.
\begin{proposition}
Let $ 0 < \alpha < 2 $. For all $ u_0 \in H^\frac{\alpha}{2}(\Omega) $ and for all $ T > 0 $, there exists a function $ u^\epsilon $ such that 
\begin{eqnarray*}
u^\epsilon \in L^\infty(0,T ; H^\frac{\alpha}{2}(\Omega)) \cap L^2(0,T ; H_N^{\alpha + 1}(\Omega))
\end{eqnarray*}
satisfying 
\begin{eqnarray} \label{eq:regeq}
\iint_Q u^\epsilon \partial_t \varphi dx dt + \iint_Q f_\epsilon(u^\epsilon) \partial_x I(u^\epsilon) \partial_x \varphi dx dt = - \int_\Omega u_0 \varphi(0,.) dx
\end{eqnarray}
for all $ \varphi \in C^1(0,T; H^1(\Omega))$ with support in $[0,T) \times \overset{-}{\Omega}.$\\
The function $ u^\epsilon $ satisfies for almost every $ t \in (0,T) $
\begin{eqnarray} \label{eq:regmass}
\int_\Omega u^\epsilon(t,x) dx = \int_\Omega u_0(x) dx 
\end{eqnarray}
and 
\begin{eqnarray} \label{eq:regine}
\| u^\epsilon(t,.) \|^2_{\overset{.}{H}^\frac{\alpha}{2}(\Omega)} + 2 \int_0^T \int_\Omega f_\epsilon(u^\epsilon) \partial_x I(u^\epsilon)^2 dx dt \leqslant \| u_0 \|^2_{\overset{.}{H}^\frac{\alpha}{2}(\Omega)}.
\end{eqnarray}
Finally, if $ \int_\Omega G_\epsilon(u_0) dx < \infty  $ then for almost every $ t \in (0,T), $
\begin{eqnarray} \label{eq:regentro}
\int_\Omega G_\epsilon(u^\epsilon(t,x)) dx + \int_0^t \| u^\epsilon(s,.) \|^2_{\overset{.}{H}_N^{\frac{\alpha}{2}+1}(\Omega)} ds \leqslant \int_\Omega G_\epsilon(u_0) dx.
\end{eqnarray}
\end{proposition}

\begin{proof}

We consider the sequence $ (u^\tau) $ constructed in Corollary \ref{cor:existence} and let $ \tau \rightarrow 0 $. Bound \eqref{eq:ineq u} and \eqref{eq:mass u} implies that $ (u^\tau) $ is bounded in $ L^\infty(0,T;H^\frac{\alpha}{2}(\Omega)) $ and $ (\partial_x I(u^\tau)) $ is bounded in $ L^2(Q) $.

\paragraph{Case $ 0 < \alpha \leqslant 1.$}

Note that 
\begin{eqnarray*}
\frac{u^\tau - S_\tau u^\tau}{\tau} = \partial_x (f_\epsilon(u^\tau) \partial_x I(u^\tau)).
\end{eqnarray*}
Since $n \geqslant 1$, the function $f_\epsilon$ is Lipschitz and so $ (f_\epsilon(u^\tau)) $ is bounded in $ L^\infty(0,T; H^\frac{\alpha}{2}(\Omega)) $ thus by the Sobolev embedding theorem, we deduce that $ (f_\epsilon(u^\tau)) $ is bounded in $ L^\infty(0,T; L^p(\Omega)) $ for all $ p < \frac{2}{1-\alpha} $. We know that $  (\partial_x I(u^\tau)) $ is bounded in $ L^2(0,T; L^2(\Omega)) $ so $ f_\epsilon(u^\tau) \partial_x I(u^\tau) $ is bounded in $ L^2(\tau,T; L^r(\Omega)) $ where $\frac{1}{r} = \frac{1}{2} + \frac{1}{p}$. We deduce that
\begin{eqnarray*}
\partial_x (f_\epsilon(u^\tau) \partial_x I(u^\tau)) \text{ is bounded in } L^2(\tau,T; W^{-1,r}(\Omega))
\end{eqnarray*}
Since $ \alpha \leqslant 1 $, we have the following embedding
\begin{eqnarray*}
H^\frac{\alpha}{2}(\Omega) \hookrightarrow L^p(\Omega) \rightarrow W^{-1,l}(\Omega)
 \end{eqnarray*}
for all $ p < \frac{2}{1-\alpha} $ and for all $l > 2$ (because $\Omega$ is bounded and we have a Sobolev space of negative regularity). Aubin's lemma implies that $ (u^\tau)$ is relatively compact in $ C^0(0,T;L^p(\Omega)) $ for all $ p < \frac{2}{1-\alpha} $.
Note that $ (\partial_x I(u^\tau)) $ is bounded in $ L^2(\Omega) $ and $ (u^\tau) $ is bounded in $ L^\infty(0,T;L^1(\Omega))$ (because $1<\frac{2}{1-\alpha}$). Hence, $ (u^\tau)$ is bounded in $ L^2(0,T;H_N^{\alpha + 1}(\Omega))$. Since
\begin{eqnarray*}
H^{\alpha+1}_N(\Omega) \hookrightarrow H^{\frac{\alpha}{2}+1}_N(\Omega) \rightarrow W^{-1,l
}(\Omega),
\end{eqnarray*}
we deduce that $ (u^\tau)$ is relatively compact in $ L^2(0,T;H^{\frac{\alpha}{2}+1}_N(\Omega))$. So we can extract a subsequence, also denoted $ (u^\tau)$, such that when $ \tau $ tends to zero we have 
\begin{align*}
& \bullet \quad u^\tau \rightarrow u^\epsilon \in L^\infty(0,T;H^\frac{\alpha}{2}(\Omega)) \text{ almost everywhere in } Q,\\
& \bullet \quad u^\tau \rightarrow u^\epsilon \text{ in } L^2(0,T;H^{\frac{\alpha}{2}+1}_N(\Omega))   \text{ strongly,}\\
& \bullet \quad \partial_x I(u^\tau) \rightharpoonup \partial_x I(u^\epsilon) \text{ in } L^2(Q)  \text{ weakly.}
\end{align*}
Now let us pass to the limit in \eqref{eq:equ}. We have 
\begin{align*}
 \iint_{Q_{\tau,T}}{\frac{u^\tau - S_\tau u^\tau}{\tau} \varphi dx dt}
 & =  \frac{1}{\tau} \bigg[ \int_0^T \int_\Omega u^\tau(t,x) \varphi(t,x) dt dx - \int_0^\tau \int_\Omega u^\tau(t,x) \varphi(t,x) dt dx \\ & - \int_0^{T-\tau} \int_\Omega u^\tau(t,x) \varphi(t + \tau,x) dt dx \bigg] \\
 & = \int_0^T \int_\Omega u^\tau(t,x) \frac{\varphi(t,x) - \varphi(t + \tau,x)}{\tau} dx dt \\ & - \frac{1}{\tau} \int_0^\tau \int_\Omega u^\tau(t,x) \varphi(t,x) dt dx + \frac{1}{\tau} \int_{T - \tau}^T \int_\Omega u^\tau(t,x) \varphi(t + \tau,x) dt dx \\
 & \underset{\tau\rightarrow 0}{\longrightarrow} - \iint_Q u^\epsilon \partial_t \varphi dx dt - \int_\Omega u^\epsilon(0,x) \varphi(0,x) dx + 0. 
 \end{align*}
 For the nonlinear term, we integrate by parts
\begin{align} \label{eq: pass}
\iint_Q f_\epsilon(u^\tau) \partial_x I(u^\tau) \partial_x \varphi = -\iint_Q f_\epsilon(u^\tau) I(u^\tau) \partial^2_{xx} \varphi - \iint_Q n (u^\tau)^{n-1} \partial_x u^\tau I(u^\tau)  \partial_x \varphi .
\end{align}
We have
\begin{eqnarray*}
u^\tau \rightarrow u^\epsilon \text{ in } L^2(0,T;H_N^s(\Omega)) \text{ for all } s < 1 + \alpha.
\end{eqnarray*}
So
\begin{eqnarray*}
I(u^\tau) \rightarrow I(u^\epsilon) \text{ in } L^2(0,T;H^{s'}(\Omega)) \text{ for all } s' < 1
\end{eqnarray*}
and
\begin{eqnarray*}
\partial_x u^\tau \rightarrow \partial_x u^\epsilon \text{ in } L^2(0,T;H^{s''}(\Omega)) \text{ for all } s'' < \alpha.
\end{eqnarray*}
So we deduce the following convergences
\begin{align*}
& I(u^\tau) \rightarrow I(u^\epsilon) \text{ in } L^2(0,T;L^q(\Omega)) \text{ for all } q < \infty. \\
 & u^\tau_x \rightarrow u^\epsilon_x \text{ in } L^2(0,T;L^p(\Omega)) \text{ for all } p < \frac{2}{1-\alpha}.
 \end{align*}
 Furthermore, since $u^\tau \rightarrow u^\epsilon$ in $C^0(0,T;L^p(\Omega))$ for all $p < \frac{2}{1-\alpha}$ and $f_\epsilon$ is lipschitz then 
 \begin{eqnarray*}
  f_\epsilon(u^\tau) \rightarrow f_\epsilon(u^\epsilon) \text{ in } C^0(0,T;L^p(\Omega)) \text{ for all } p < \frac{2}{1-\alpha}.
 \end{eqnarray*}
 For the term $(u^\tau)^{n-1}$, if $n\geqslant 2$ then the function $s \rightarrow s^{n-1}$ is lipschitz and 
 \begin{eqnarray*}
 (u^\tau)^{n-1} \rightarrow (u^\epsilon)^{n-1} \text{ in } C^0(0,T;L^p(\Omega)) \text{ for all } p < \frac{2}{1-\alpha}.
\end{eqnarray*}
If $n < 2$ then $\frac{p}{n-1} \geqslant 1$ and
\begin{eqnarray*}
(u^\tau)^{n-1} \rightarrow (u^\epsilon)^{n-1} \text{ in } C^0(0,T;L^{\frac{p}{n-1}}(\Omega)) \text{ for all } p < \frac{2}{1-\alpha}.
\end{eqnarray*}
 Thus we can pass to the limit in \eqref{eq: pass} and reverse the integration by parts to obtain
 \begin{align*}
 \iint_Q f_\epsilon(u^\tau) \partial_x I(u^\tau) \partial_x \varphi \rightarrow & - \iint_Q f_\epsilon(u^\epsilon) I(u^\epsilon) \partial^2_{xx} \varphi - \iint_Q n (u^\epsilon)^{n-1} \partial_x u^\epsilon I(u^\epsilon) \partial_x \varphi \\ & =  \iint_Q f_\epsilon(u^\epsilon) \partial_x I(u^\epsilon) \partial_x \varphi.
  \end{align*}
For the properties of $ u^\epsilon$, first since $ u^\tau \rightarrow u^\epsilon $ in $ L^\infty(0,T; L^1(\Omega)) $ mass conservation equation~\eqref{eq:regmass} follows from \eqref{eq:mass u}.

Secondly, we note that $ (u^\tau) $ is bounded in $ L^\infty(0,T; H^\frac{\alpha}{2}(\Omega)) $ so $ (u^\tau) $ weakly converges to $ u^\epsilon $ in $ H^\frac{\alpha}{2}(\Omega) $ and
\begin{eqnarray*}
\| u^\epsilon \|_{\overset{.}{H}^\frac{\alpha}{2}(\Omega)} \leq \liminf_{\tau \rightarrow 0} \| u^\tau \|_{\overset{.}{H}^\frac{\alpha}{2}(\Omega)}.
\end{eqnarray*}
 Note that estimate \eqref{eq:ineq u} implies that $ \sqrt{f_\epsilon(u^\tau)} \partial_x I(u^\tau) $ is bounded in $ L^2(0,T;L^2(\Omega))$ thus it weakly converges in $ L^2(0,T;L^2(\Omega))$ and the lower semicontinuity permits us to conclude \eqref{eq:regine}.

Finally, to derive \eqref{eq:regentro} we note that $ G_\epsilon(u^\tau) \rightarrow G_\epsilon(u^\epsilon) $ almost everywhere and Fatou's lemma implies for almost every $ t \in (0,T) $
\begin{eqnarray*}
\int_\Omega G_\epsilon(u^\epsilon(t,x)) dx \leq \liminf_{\tau \rightarrow 0} \int_\Omega G_\epsilon(u^\tau(t,x)) dx.
\end{eqnarray*} 
Furthermore, $ (u^\tau)$ is relatively compact in $ L^2(0,T;H^{\frac{\alpha}{2} +1}_N(\Omega)) $ thus
\begin{eqnarray*}
\int_0^t \| u^\epsilon(s) \|^2_{\overset{.}{H}_N^{\frac{\alpha}{2} + 1}} ds = \lim_{\tau \rightarrow 0} \int_0^t \| u^\tau(s) \|^2_{\overset{.}{H}_N^{\frac{\alpha}{2} + 1}} ds.
\end{eqnarray*} 
Hence \eqref{eq: entro u} implies \eqref{eq:regentro}. 

\paragraph{Case $ 1 < \alpha < 2.$}

Note that 
\begin{eqnarray*}
\frac{u^\tau - S_\tau u^\tau}{\tau} = \partial_x (f_\epsilon(u^\tau) \partial_x I(u^\tau)).
\end{eqnarray*}
We have $ (u^\tau)$ is bounded in $ L^\infty(0,T; H^\frac{\alpha}{2}(\Omega)) $ so by the Sobolev embedding theorem, we deduce that $ (u^\tau) $ is bounded in $ L^\infty(0,T; C^{0,\frac{\alpha-1}{2}}(\Omega)) $. Thus $ (f_\epsilon(u^\tau)) $ is bounded in $ L^\infty(0,T; L^\infty(\Omega)) $. We know that $ (\partial_x I(u^\tau))$ is bounded in $ L^2(0,T; L^2(\Omega)) $ so $ (f_\epsilon(u^\tau) \partial_x I(u^\tau)) $ is bounded in $ L^2(\tau,T; L^2(\Omega)) $. We deduce that
\begin{eqnarray*}
\partial_x (f_\epsilon(u^\tau) \partial_x I(u^\tau)) \text{ is bounded in } L^2(\tau,T; W^{-1,2}(\Omega)).
\end{eqnarray*}
Since $ \alpha > 1 $ we have the following embedding
 \begin{eqnarray*}
  H^\frac{\alpha}{2}(\Omega) \hookrightarrow C^{0,\frac{\alpha - 1}{2}}(\Omega) \rightarrow W^{-1,2}(\Omega).
 \end{eqnarray*}
 Aubin's lemma implies that the sequence $ (u^\tau)$ is relatively compact in $ C^0(0,T;C^{0,\frac{\alpha-1}{2}}(\Omega)) $. Since $ (\partial_x I(u^\tau)) $ is bounded in $ L^2(\Omega) $ and $ (u^\tau)$ is bounded in $ L^\infty(0,T;L^1(\Omega))$ then, $ (u^\tau)$ is bounded in $ L^2(0,T;H_N^{\alpha + 1}(\Omega))$. Using the following embedding
\begin{eqnarray*}
H^{\alpha+1}_N(\Omega) \hookrightarrow H^{\frac{\alpha}{2}+1}_N(\Omega) \rightarrow W^{-1,2}(\Omega),
\end{eqnarray*}
we deduce that $ (u^\tau)$ is relatively compact in $ L^2(0,T;H^{\frac{\alpha}{2}+1}_N(\Omega))$. So for a subsequence we have 
 \begin{align*}
& \bullet \quad u^\tau \rightarrow u^\epsilon  \text{ locally uniformly, } \\
&  \bullet \quad  \partial_x I(u^\tau) \rightharpoonup \partial_x I(u^\epsilon) \text{ in } L^2(Q)  \text{-weakly,} \\
 & \bullet \quad u^\tau \rightarrow u^\epsilon \text{ in } L^2(0,T;H^{\frac{\alpha}{2}+1}_N(\Omega))   \text{ strongly.}
 \end{align*}
 Let us pass to the limit in \eqref{eq:equ}. As in the first case
 \begin{align*}
  \iint_{Q_{\tau,T}}{\frac{u^\tau - S_\tau u^\tau}{\tau} \varphi dx dt}  \underset{\tau\rightarrow 0}{\longrightarrow} - \iint_Q u^\epsilon \partial_t \varphi dx dt - \int_\Omega u^\epsilon(0,x) \varphi(0,x) dx .
 \end{align*}
For the nonlinear term, since
\begin{eqnarray*}
u^\tau \rightarrow u^\epsilon  \text{ locally uniformly,}
\end{eqnarray*} 
Then
\begin{eqnarray*}
f_\epsilon(u^\tau) \partial_x \varphi \rightarrow f_\epsilon(u^\epsilon) \partial_x \varphi \text{ in } L^2(0,T;L^2(\Omega))- \text{strongly.}
\end{eqnarray*}
Furthermore
\begin{eqnarray*}
\partial_x I(u^\tau) \rightharpoonup \partial_x I(u^\epsilon)  \text{ in } L^2(0,T;L^2(\Omega))- \text{weakly.}
\end{eqnarray*}
Hence
\begin{eqnarray*}
\iint_{Q}{f_\varepsilon(u^\tau)\partial_x I(u^\tau) \partial_x \varphi dx dt}\rightarrow \iint_Q{f_\epsilon(u^\epsilon)\partial_x I(u^\epsilon) \partial_x \varphi dx dt}
\end{eqnarray*}
and the proof is complete.

For the properties of $ u^\epsilon$, the proofs of estimates \eqref{eq:regmass}, \eqref{eq:regentro} and \eqref{eq:regine} are the same as in the first case.
\end{proof}

\section{Proofs of main results}

\subsection{Proof of Theorem \ref{th:th1}}

Consider the sequence $ (u^\epsilon)$ such that $ u^\epsilon \in L^\infty(0,T; H^\frac{\alpha}{2}(\Omega)) \cap L^2(0,T; H^{\frac{\alpha}{2}+ 1}(\Omega)) $ solution of \eqref{pro:reg}. Our goal is to pass to the limit $ \epsilon \rightarrow 0 $. 

Note that \eqref{eq:regine} and \eqref{eq:regmass} imply that $ (u^\epsilon)$ is bounded in $ L^\infty(0,T; H^\frac{\alpha}{2}(\Omega)) $. Since $f_\epsilon$ is Lipschitz then $f_\epsilon(u^\epsilon)$ is bounded in $ L^\infty(0,T; H^\frac{\alpha}{2}(\Omega)) $. So by using the Sobolev embedding theorem, we deduce that $ f_\epsilon(u^\epsilon)$ is bounded in $ L^\infty(0,T; L^p(\Omega))$ for all $ p < \frac{2}{1-\alpha} $.

Furthermore, \eqref{eq:regine} also implies that $ f_\epsilon(u^\epsilon)^\frac{1}{2} \partial_x I(u^\epsilon) $ is bounded in $ L^2(0,T; L^2(\Omega)) $. Thus
\begin{eqnarray*}
f_\epsilon(u^\epsilon) \partial_x I(u^\epsilon) \text{is bounded in } L^2(0,T; L^r(\Omega)) \text{ where } \frac{1}{r} = \frac{1}{2} + \frac{1}{2p}.
\end{eqnarray*}
Hence
\begin{eqnarray*}
\partial_t u^\epsilon = - \partial_x (f_\epsilon(u^\epsilon) \partial_x I(u^\epsilon)) \text{ is bounded in } L^2(0,T;W^{-1,r}(\Omega)).
\end{eqnarray*} 
Since 
\begin{eqnarray*}
 H^\frac{\alpha}{2}(\Omega) \hookrightarrow L^\frac{2}{1-\alpha}(\Omega) \rightarrow W^{-1,l}(\Omega),
\end{eqnarray*}
Aubin's lemma implies that $ (u^\epsilon)$ is relatively compact in $ C^0(0,T; L^p(\Omega)) $ for all $ p < \frac{2}{1-\alpha} $. So we can extract a subsequence such that \\
$ \bullet \quad u^\epsilon \rightarrow u  $ in $ C^0(0,T; L^p(\Omega)) $ for all $ p < \frac{2}{1-\alpha}. $ \\
$ \bullet \quad u^\epsilon \rightarrow u \in L^\infty(0,T; H^\frac{\alpha}{2}(\Omega)) $ almost everywhere in Q. 

Let us pass to the limit in \eqref{eq:regeq}. Let $ \varphi \in D( [0,T) \times \bar{\Omega}) $ satisfying $ \partial_x \varphi = 0 $ on $(0,T) \times \partial\Omega $. Since $ u^\epsilon \rightarrow u $ in $ C^0(0,T; L^1(\Omega)) $, we have
\begin{eqnarray*}
\iint_Q u^\epsilon \partial_t \varphi dx dt \rightarrow \iint_Q u \partial_t \varphi dx dt.
\end{eqnarray*}
Remark that \eqref{eq:regine} implies that 
\begin{eqnarray*}
\epsilon \iint{(\partial_x I(u^\epsilon))^2} \leq c.
\end{eqnarray*}
The Cauchy-Schwarz inequality implies
\begin{eqnarray*}
\epsilon \iint \partial_x I(u^\epsilon) \partial_x \varphi dx dt \leqslant c(\varphi) \sqrt{\epsilon} (\sqrt{\epsilon} \|\partial_x I(u^\epsilon)\|_2) \rightarrow 0.
\end{eqnarray*}
Estimate \eqref{eq:regine} also gives  that $ (u^\epsilon)_+^\frac{n}{2} \partial_x I(u^\epsilon) $ is bounded in $ L^2(0,T; L^2(\Omega)) $. For the term $(u^\epsilon)^\frac{n}{2}$, we consider two cases, if $n \geqslant 2$ then the function $s \rightarrow s^\frac{n}{2}$ is Lipschitz and $((u^\epsilon)^\frac{n}{2})$ is bounded in $ L^\infty(0,T; L^p(\Omega)) $ for all $ p < \frac{2}{1-\alpha} $. We deduce that $ ((u^\epsilon)_+^n \partial_x I(u^\epsilon)) $ is bounded in $ L^2(0,T; L^m(\Omega)) $ where $\frac{1}{m} = \frac{1}{2} + \frac{1}{p}$. If $n<2$ then $ ((u^\epsilon)^\frac{n}{2}) $ is bounded in $ L^\infty(0,T; L^\frac{2p}{n}(\Omega)) $ for all $ p < \frac{2}{1-\alpha} $ (in this case $\frac{2p}{n} \geqslant 1$). We deduce that $ ((u^\epsilon)_+^n \partial_x I(u^\epsilon)) $ is bounded in $ L^2(0,T; L^m(\Omega)) $ where $\frac{1}{m} = \frac{1}{2} + \frac{n}{2p}$, hence
\begin{eqnarray*}
h^\epsilon := (u^\epsilon)_+^n \partial_x I(u^\epsilon) \rightharpoonup h \text{ in } L^2(0,T;L^m(\Omega))\text{ weakly.}
\end{eqnarray*}
Passing to the limit we obtain
\begin{eqnarray*}
\iint_Q {u \partial_t \varphi dx dt} + \iint_Q {h \partial_x \varphi dx dt} = - \int_\Omega{u_0 \varphi(0,x) dx}.
\end{eqnarray*}
It remains to show that 
\begin{eqnarray*}
h = u_+^n \partial_x I(u)
\end{eqnarray*}
in the following sense
\begin{eqnarray} \label{eq:h}
\iint_Q h \varphi dx dt = -\iint_Q n u_+^{n-1} \partial_x u I(u) \varphi dx dt -\iint_Q u_+^n I(u) \partial_x \varphi dx dt
\end{eqnarray}
for all test functions $ \varphi $ such that $ \varphi = 0 $ on $ (0,T)\times\partial\Omega $, that is
\begin{eqnarray*}
h = \partial_x(u_+^n I(u)) - n u_+^{n-1} \partial_x u I(u)  \text{ in } D'(\Omega).
\end{eqnarray*}
Note that $ G_\epsilon $ is decreasing with respect to $\epsilon$, so
\begin{eqnarray*}
\int_\Omega G_\epsilon(u_0) dx \leq \int_\Omega G(u_0) dx \leq c.
\end{eqnarray*}
Thus estimate \eqref{eq:regentro} implies that $ (u^\epsilon)$ is bounded in $ L^2(0,T; H_N^{\frac{\alpha}{2}+1}(\Omega)) $. Recall that $ ( \partial_t u^\epsilon)$ is bounded in $ L^2(0,T; W^{-1,l}(\Omega)) $.
Aubin's lemma implies that 
\begin{eqnarray*}
(u^\epsilon)\text{ is relatively compact in } L^2(0,T; H_N^s(\Omega)) \text{ for all } s < \frac{\alpha}{2} + 1.
\end{eqnarray*}
Hence
\begin{eqnarray*}
( \partial_x u^\epsilon)\text{ is relatively compact in } L^2(0,T; H^{s'}(\Omega)) \text{ for all } s' < \frac{\alpha}{2}
\end{eqnarray*}
and 
\begin{eqnarray*}
(I(u^\epsilon))\text{ is relatively compact in } L^2(0,T; H_N^{s''}(\Omega)) \text{ for all } s'' < 1 - \frac{\alpha}{2}.
\end{eqnarray*}
Thus we can extract a subsequence such that 
\begin{align*}
& u^\epsilon \rightarrow u \text{ in } C^0(0,T; L^p(\Omega)) \text{ for all } p < \frac{2}{1-\alpha}, \\
& I(u^\epsilon) \rightarrow I(u) \text{ in } L^2(0,T; L^q(\Omega)) \text{ for all } q < \infty, \\
& \partial_x u^\epsilon \rightarrow \partial_x u \text{ in } L^2(0,T; L^p(\Omega)) \text{ for all } p < \frac{2}{1-\alpha}.
\end{align*}
We write
\begin{align*}
\iint_Q{h_\epsilon \varphi} dx dt & =\iint{(u^\epsilon)_+^n \partial_x I(u^\epsilon) \varphi } dx dt \\
& = - \iint{n (u^\epsilon)_+^{n-1} \partial_x u^\epsilon I(u^\epsilon) \varphi} dx dt - \iint{(u^\epsilon)_+^n I(u^\epsilon) \partial_x \varphi}  dx dt .
\end{align*}
Using these convergences and the fact that $ I(u^\epsilon) $ converges in $ L^2(0,T; L^q(\Omega)) $ for all $ q < \infty $ we can pass to the limit and obtain \eqref{eq:h}.
Note that for the terms $(u^\epsilon)^n$ and $(u^\epsilon)^{n-1}$ we consider two cases $n\geqslant2$ and $n<2$ and we proceed as above. In the first case the functions $s \rightarrow s^n$ and $s \rightarrow s^{n-1}$ are Lipschitz and then $(u^\epsilon)^n \rightarrow u^n$ and $(u^\epsilon)^{n-1} \rightarrow u^{n-1}$ in $C^0(0,T; L^p(\Omega))$ for all $p < \frac{2}{1-\alpha}$. If $n<2$ then $\frac{p}{n} \geqslant 1$ and $\frac{p}{n-1} \geqslant 1$ and $(u^\epsilon)^n \rightarrow u^n$ in $C^0(0,T; L^\frac{p}{n}(\Omega))$ and $(u^\epsilon)^{n-1} \rightarrow u^{n-1}$ in $C^0(0,T; L^\frac{p}{n-1}(\Omega))$ for all $p < \frac{2}{1-\alpha}$.

For the properties of $ u $, passing to the limit in \eqref{eq:regmass} implies mass conservation equation \eqref{masscons}.

Since $ (u^\epsilon)$ is bounded in $ L^2(0,T; H_N^{\frac{\alpha}{2}+1}(\Omega)) $ then $ u^\epsilon \rightharpoonup u $ and
\begin{eqnarray*}
\| u \|_{L^2(0,T;H_N^{\frac{\alpha}{2}+1}(\Omega))} \leq \liminf_{\epsilon \rightarrow 0} \| u^\epsilon \|_{L^2(0,T;H_N^{\frac{\alpha}{2}+1}(\Omega))}.
\end{eqnarray*}
Note that
\begin{eqnarray*}
G_\epsilon(u^\epsilon) \rightarrow G(u) \text{ almost everywhere and }  G_\epsilon(u_0) \leq G (u_0).
\end{eqnarray*}
Then by Fatou's lemma estimate \eqref{eq:entropy} follows from \eqref{eq:regentro}.

Remark that estimate \eqref{eq:regine} implies that $ g_\epsilon = (u^\epsilon)_+^\frac{n}{2} \partial_x I(u^\epsilon) $ weakly converges in $ L^2(Q) $ to a function $ g $ and the lower semi-continuity of the norm implies \eqref{eq:ineq}. It remains to prove that 
\begin{eqnarray} \label{eq: g}
g = \partial_x (u_+^\frac{n}{2} I(u)) - \frac{n}{2} u_+^{\frac{n}{2}-1} \partial_x u I(u) \text{ in } D'(\Omega).
\end{eqnarray}
We have
\begin{align*}
\iint_Q{g_\epsilon \varphi} dx dt & =\iint{(u^\epsilon)_+^\frac{n}{2} \partial_x I(u^\epsilon) \varphi } dx dt \\
& = - \iint{\frac{n}{2} (u^\epsilon)_+^{\frac{n}{2}-1} \partial_x u^\epsilon I(u^\epsilon) \varphi} dx dt - \iint{(u^\epsilon)_+^\frac{n}{2} I(u^\epsilon) \partial_x \varphi} dx dt.
\end{align*}
Also, using the convergences above and the fact that  $ I(u^\epsilon) \rightarrow I(u) $ in $ L^2(0,T;L^q(\Omega)) $ for all $ q < \infty $, we can pass to the limit and obtain \eqref{eq: g}.
Note also that for the terms $(u^\epsilon)^{\frac{n}{2}-1}$ and $(u^\epsilon)_+^\frac{n}{2}$ we consider two cases $n \geqslant 4$ and $n < 4$ and we proceed as above.

It remains to prove that $ u $ is a nonnegative function. Note that estimate \eqref{eq:regentro} implies that for all $ t \in (0,T) $ 
\begin{eqnarray*}
\int_\Omega G_\epsilon(u^\epsilon(t,x)) dx \leq \int_\Omega G_\epsilon(u_0(t,x)) dx.
\end{eqnarray*}
Since 
\begin{eqnarray*}
\int_\Omega G_\epsilon(u_0(t,x)) dx \leq \int_\Omega G(u_0(t,x)) dx < \infty,
\end{eqnarray*}
we conclude that 
\begin{eqnarray} \label{eq: contra}
\limsup_{\epsilon \rightarrow 0} \int_\Omega G_\epsilon(u^\epsilon(t,x)) dx < \infty.
\end{eqnarray}
Note that for all $ \delta > 0, $
\begin{eqnarray*}
\lim_{\epsilon \rightarrow 0} G_\epsilon(-\delta) = +\infty.
\end{eqnarray*}
Recall that $ u^\epsilon(t,.) $ converges almost everywhere. So for $ \eta > 0 $, Egorov's theorem implies the existence of a set $ A_\eta \subset \Omega $ such that 
\begin{eqnarray*}
\mid \Omega\setminus A_\eta \mid \leq \eta \text{ and } u^\epsilon(t,.) \rightarrow u(t,.) \text{ uniformly in } A_\eta.
\end{eqnarray*}
Let $ \delta > 0 $. We consider
\begin{eqnarray*}
C_{\eta,\delta} = A_\eta \cap \{u(t,.) \leq -2\delta\}.
\end{eqnarray*}
For every $ \eta , \delta > 0 $, there exists $ \epsilon_0(\eta,\delta) $ such that if $ \epsilon \leq \epsilon_0(\eta,\delta) $ then $ u^\epsilon(t,.) \leq -\delta $ in $ C_{\eta,\delta} $.

This implies that $ C_{\eta,\delta} $ has measure zero. Indeed, if not then for $ \epsilon \leq \epsilon_0(\eta,\delta) $ we have
\begin{eqnarray*}
G_\epsilon(u^\epsilon(t,x)) \geq G_\epsilon(-\delta) \rightarrow +\infty.
\end{eqnarray*}
By Fatou's lemma
\begin{eqnarray*}
\liminf_{\epsilon \rightarrow 0}\int_{C_{\eta,\delta}}{G_\epsilon(u^\epsilon(t,x)) dx} \geqslant \int_{C_{\eta,\delta}} \liminf_{\epsilon \rightarrow 0} G_\epsilon(u^\epsilon(t,x)) dx = +\infty
\end{eqnarray*}
which contradicts \eqref{eq: contra}.

Hence for all $ \delta > 0 $ and all $ \eta > 0 $, we have
\begin{eqnarray*}
\mid \{u(t,.)\leq -2\delta \}\mid \leq \mid C_{\eta,\delta} \mid + \mid \Omega \setminus A_\eta \mid \leq \eta.
\end{eqnarray*}

Thus, $ \mid \{u(t,.)\leq -2\delta \}\mid = 0 $ for all $ \delta > 0 $. We conclude that
\begin{eqnarray*}
\left\{u(t,.) < 0 \right\} = \underset{k \geq 1}{\bigcup} \left\{ u(t,.) \leq \frac{-1}{k} \right\}
\end{eqnarray*}
has measure zero and so $ u(t,x) \geq 0 $ for almost every $ x \in \Omega $ and for all $ t > 0 $. 
\subsection{Proof of Theorem \ref{th:th2}}

We organize this proof in two stages. In the first stage we consider nonnegative $u_0 \in H^\frac{\alpha}{2}(\Omega)$ satisfying \eqref{entropy condition} and we prove the existence of solutions as in Theorem \ref{th:th1}. In the second stage we use this information to prove the existence of solutions for nonnegative initial data which belongs to $H^\frac{\alpha}{2}(\Omega).$

\paragraph*{First stage} Consider the sequence $ (u^\epsilon)$ such that $ u^\epsilon \in L^\infty(0,T; H^\frac{\alpha}{2}(\Omega)) \cap L^2(0,T; H_N^{\frac{\alpha}{2}+ 1}(\Omega)) $ solution of \eqref{pro:reg}. Our goal is to pass to the limit $ \epsilon \rightarrow 0 $.  Note that \eqref{eq:regine} implies that $ (u^\epsilon)$ is bounded in $ L^\infty(0,T; H^\frac{\alpha}{2}(\Omega))$. So by using the Sobolev embedding theorem, we deduce that $ (u^\epsilon)$ is bounded in $ L^\infty(0,T; C^{0,\frac{\alpha-1}{2}}(\Omega))$. Hence $ (f_\epsilon(u^\epsilon))$ is bounded in $ L^\infty(0,T; L^\infty(\Omega))$. Furthermore \eqref{eq:regine} gives that $ (f_\epsilon(u^\epsilon)^\frac{1}{2} \partial_x I(u^\epsilon))$ is bounded in $ L^2(0,T; L^2(\Omega)) $. We deduce that
\begin{eqnarray*}
(f_\epsilon(u^\epsilon) \partial_x I(u^\epsilon)) \text{ is bounded in } L^2(0,T; L^2(\Omega)).
\end{eqnarray*}
So
\begin{eqnarray*}
\partial_t u^\epsilon = - \partial_x (f_\epsilon(u^\epsilon) \partial_x I(u^\epsilon)) \text{ is bounded in } L^2(0,T; W^{-1,2}(\Omega)).
\end{eqnarray*}
Since
\begin{align*}
H^\frac{\alpha}{2}(\Omega) \hookrightarrow C^{0,\frac{\alpha - 1}{2}}(\Omega) \rightarrow W^{-1,2}(\Omega),
\end{align*} 
Aubin's lemma implies that $ (u^\epsilon)$ is relatively compact in $ C^0(0,T;C^{0,\frac{\alpha-1}{2}}(\Omega)) $. So we can extract a subsequence such that
\begin{eqnarray*}
u^\epsilon \rightarrow u  \text{ locally uniformly. }
\end{eqnarray*}
Now let us pass to the limit in \eqref{eq:regeq}. Proceeding as in the case $0<\alpha\leqslant1$ we get the same results but it remains to prove the equation on $h$ i.e. \eqref{eq:h}. Since
\begin{eqnarray*}
\int_\Omega G_\epsilon(u_0) dx \leq \int_\Omega G(u_0) dx \leq c,
\end{eqnarray*}
estimate \eqref{eq:regentro} implies that $ (u^\epsilon)$ is bounded in $ L^2(0,T; H_N^{\frac{\alpha}{2}+1}(\Omega)) $. Recall that $ ( \partial_t u^\epsilon)$ is bounded in $ L^2(0,T; W^{-1,2}(\Omega)) $. Aubin's lemma implies that 
\begin{eqnarray*}
(u^\epsilon) \text{ is relatively compact in } L^2(0,T; H^s(\Omega)) \text{ for all } s < \frac{\alpha}{2} + 1.
\end{eqnarray*}
Hence
\begin{eqnarray*}
( \partial_x u^\epsilon) \text{ is relatively compact in } L^2(0,T; H^{s'}(\Omega)) \text{ for all } s' < \frac{\alpha}{2}
\end{eqnarray*}
and since $ \alpha < 2 $ then $ (u^\epsilon)$  is relatively compact in $ L^2(0,T; H^\alpha(\Omega)) $ and
\begin{eqnarray*}
(I(u^\epsilon)) \text{ is relatively compact in } L^2(0,T; L^2(\Omega)).
\end{eqnarray*}
Thus we can extract a subsequence such that 
\begin{align*}
& u^\epsilon \rightarrow u \text{ locally uniformly, } \\
& I(u^\epsilon) \rightarrow I(u) \text{ in } L^2(0,T; L^2(\Omega)), \\
& \partial_x u^\epsilon \rightarrow \partial_x u \text{ in } L^2(0,T; \text{locally uniformly with respect to x}).
\end{align*}
Using an integration by parts for the equation of $h_\epsilon$ and using these convergences we can pass to the limit and obtain \eqref{eq:h}.

For the properties of $ u $, the proofs are the same as in the case $0<\alpha\leqslant1$  but we use these convergences above to obtain the equation on $ g $. 

We prove also that $ u $ is a nonnegative function as in the case $0<\alpha\leqslant1$.

\paragraph*{Second stage} Now we consider the case where $u_0 \geqslant 0 $ belongs to $ H^\frac{\alpha}{2}(\Omega)$ without the additional condition \eqref{entropy condition}. If we define 
\begin{eqnarray*}
u_{0 \delta}(x) = u_0(x) + \delta 
\end{eqnarray*}
and denote $u_\delta$ the nonnegative solution $u$ constructed in the first stage for the initial data $u_{0\delta}$, which satisfies \eqref{entropy condition}, then $u_\delta$ satisfies
\begin{eqnarray} \label{inequalities}
\mid u_\delta \mid \leq A, \quad \iint_Q u_\delta^n \partial_x I(u_\delta)^2 dx dt \leq C, \quad \mid u_\delta(t,x_2) - u_\delta(t,x_1) \mid \leq k \mid x_2 - x_1 \mid^\frac{\alpha -1}{2},
\end{eqnarray}
with constants $C, A, K$ independent of $\delta$ and $T$.

\begin{proposition} \label{prop:holder}
There exists a constant $M$ independent of $\delta$ and $T$ such that
\begin{eqnarray} \label{M}
\mid u_\delta(t_2,x) - u_\delta(t_1,x) \mid \leqslant M \mid t_2 - t_1 \mid^\frac{\alpha -1}{2(\alpha +2)}
\end{eqnarray}
for all $x \in \Omega$, $t_1$ and $t_2 \in (0,T)$.
\end{proposition}

\begin{proof}
The proof is given in Appendix A.
\end{proof}

Taking a subsequence 
\begin{eqnarray*}
u_\delta \rightarrow u \quad \text{ locally uniformly in } Q,
\end{eqnarray*}
we will prove Theorem \ref{th:th2}. Let $ \varphi \in \mathcal{D}([0,T)\times \overset{-}{\Omega}) $ satisfying $ \partial_x \varphi = 0 $ on $(0,T) \times \partial\Omega $. We have
\begin{eqnarray} \label{eq:udelta}
\iint_Q u_\delta \partial_t \varphi dt dx + \iint_Q u_\delta^n \partial_x I(u_\delta) \partial_x \varphi dx dt = -\int_\Omega (u_0 + \delta) \varphi(0,.) dx.
\end{eqnarray}
Since $u_\delta \rightarrow u$ locally uniformly then
\begin{eqnarray} \label{limit}
\iint_Q u_\delta \partial_t \varphi dt dx \rightarrow \iint_Q u \partial_t \varphi dt dx \quad \text{ and } \quad \int_\Omega (u_0 + \delta) \varphi(0,.) dx \rightarrow \int_\Omega u_0 \varphi(0,.) dx \quad \text{ as } \delta \rightarrow 0.
\end{eqnarray} 
It remains to pass to the limit in the nonlinear term. We consider 
\begin{eqnarray*}
h_\delta = u_\delta^n \partial_x I(u_\delta).
\end{eqnarray*}
From \eqref{inequalities}, $((u_\delta)^\frac{n}{2} \partial_x I(u_\delta))$ is bounded in $L^2(Q)$ and since $(u_\delta)$ is bounded in $L^\infty(Q)$ so $(h_\delta)$ is bounded in $L^2(Q)$ and weakly converges to $h$ in $L^2(Q)$. Our aim is to prove that
\begin{eqnarray*}
h = \begin{cases} u^n \partial_x I(u) \quad & \text{in } Q_+:= \{ u > 0\} \cap Q \\
0 \quad & \text{elsewhere}. 
\end{cases}
\end{eqnarray*}
For any $\eta > 0$ we have
\begin{eqnarray*}
c \geqslant \int_{\{u \geqslant \eta \}\cap Q} u_\delta^n \partial_x I(u_\delta)^2 \geqslant \left( \frac{\eta}{2}\right)^n \int_{\{u \geqslant \eta\}\cap Q} \partial_x I(u_\delta)^2,
\end{eqnarray*}
so $(\partial_x I(u_\delta)) $ is bounded in $L^2(\{u \geqslant \eta\}\cap Q).$ Thus for all $k \in \mathbb{N}$,
\begin{eqnarray*}
(\partial_x I(u_\delta)) \text{ weakly converges in  } L^2(Q_k)
\end{eqnarray*}
where $Q_k := \{u \geqslant \frac{1}{k}\} \cap Q $. So, up to a subsequence,
\begin{eqnarray*}
\partial_x I(u_\delta)) \text{ weakly converges to } g \text{ in } L^2_{loc}(Q_+)
\end{eqnarray*}
where $Q_+ =\underset{k\in\mathbb{N}}{\bigcup} P_k =\{u>0\} \cap Q$. This implies that
\begin{eqnarray*}
\partial_x I(u_\delta) \rightarrow g \quad \text{in }\mathcal{D}'(Q_+).
\end{eqnarray*}
It remains to prove that
\begin{eqnarray*}
g = \partial_x I(u)\quad \text{in }\mathcal{D}'(Q_+).
\end{eqnarray*}
Since $u_\delta \rightarrow u$ locally uniformly in $Q$ then by using Corollary \ref{corollary}
\begin{eqnarray*}
I(u_\delta) \rightarrow I(u) \text{ in } \mathcal{D}'(Q).
\end{eqnarray*}
So, $\partial_x I(u_\delta) \rightarrow \partial_x I(u)$ in $\mathcal{D}'(Q)$.
Now, let $\varphi \in \mathcal{D}(Q_+)$ we have
\begin{eqnarray*}
\langle \partial_x I(u_\delta), \varphi\rangle_{\mathcal{D}'(Q_+)\mathcal{D}(Q_+)} = \langle \partial_x I(u_\delta), \bar{\varphi} \rangle_{\mathcal{D}'(Q)\mathcal{D}(Q)} \underset{\delta\rightarrow 0}{\longrightarrow} \langle \partial_x I(u), \bar{\varphi} \rangle_{\mathcal{D}'(Q)\mathcal{D}(Q)}
\end{eqnarray*}
where $\bar{\varphi} $ is the extension by $0$ of $\varphi$ to $Q$. So
\begin{eqnarray} \label{conv}
g = \partial_x I(u)  \text{ in }\mathcal{D}'(Q_+) \text{ and } \partial_x I(u) \in L^2_{loc}(Q_+).
\end{eqnarray}
On the other hand, if $\delta$ is sufficiently small, then
\begin{eqnarray} \label{conv0}
\mid \iint_{\{u=0 \} \cap Q} u_\delta^n \partial_x I(u_\delta) \partial_x \varphi \mid \leqslant c \delta^\frac{n}{2} (\iint u_\delta^n \partial_x I(u_\delta)^2)^\frac{1}{2} \leqslant C \delta^\frac{n}{2}
\end{eqnarray}
Taking $\delta \rightarrow 0$ in \eqref{eq:udelta} and using \eqref{limit}, \eqref{conv} and \eqref{conv0} we deduce that \eqref{limitP} is satisfied.
Finally since $u_\delta$ satisfies mass conservation and uniformly converges to $u$ then $u$ inherits the same property.

\subsection{Proof of Theorem \ref{th:th3}}

Consider the sequence $ (u^\epsilon) $ such that $ u^\epsilon $ solution of \eqref{pro:reg} introduced in the proof of Theorem \ref{th:th1}. Recall that (28) implies that $ (u^\epsilon) $ is bounded in $ L^2(0,T; H_N^{\frac{\alpha}{2}+1}(\Omega)). $

\paragraph{Case $ 0 < \alpha < 1.$}

We recall that $ (\partial_t u^\epsilon) $ is bounded in $ L^2(0,T; W^{-1,l}(\Omega)) $ . So Aubin's lemma implies that $ (u^\epsilon)$ converges in $ L^2(0,T; C^{0,\beta}(\Omega)) $ for all $ \beta < \frac{\alpha+1}{2}$. We can thus find a subsequence, also denoted $ (u^\epsilon)$, and a set $ P \subset (0,T) $ such that $ \mid (0,T) \setminus P \mid = 0 $ and for all $ t \in P $, $ u^\epsilon(t,.) $ converges strongly in $ C^\beta(\Omega) $.\\

We note that for all $ t \in P$, $ u $ is strictly positive. Indeed if there exists $ (t_0,x_0) \in P \times \Omega  $ such that $ u(t_0,x_0) = 0 $ then for any $ \beta < \frac{\alpha+1}{2} $ there exists a constant $ c_\beta $ such that for all $ x \in \Omega $
\begin{eqnarray*}
u(t_0,x) \leq c_\beta \mid x - x_0 \mid^\beta.
\end{eqnarray*}
Thus
\begin{eqnarray*}
\int G(u(t_0,x)) dx \geq \int \frac{1}{(c_\beta \mid x - x_0 \mid^\beta)^{n-2}} dx.
\end{eqnarray*}
Given $ n > 4$, we can choose $ \beta < \frac{\alpha +1}{2} $ such that $ \beta (n-2) > 1 $. We deduce
\begin{eqnarray*}
\int G(u(x,t_0)) dx = \infty
\end{eqnarray*}
which contradicts \eqref{eq: contra}.

We deduce that there exists $ \delta > 0 $ (depending on $ t $) such that for $ \epsilon $ small enough 
\begin{eqnarray*}
u^\epsilon(t,.) \geqslant \delta \text{ in } \Omega.
\end{eqnarray*}
Note that 
\begin{eqnarray*}
\liminf_{\epsilon \rightarrow 0}\int_\Omega f_\epsilon(u^\epsilon) \mid \partial_x I(u^\epsilon) \mid^2 dx < \infty \text{ for all } t \in P.
\end{eqnarray*}
Indeed, if we denote 
\begin{eqnarray*}
A_k = \{ t \in P ; \liminf_{\epsilon \rightarrow 0}\int_\Omega f_\epsilon(u^\epsilon) \mid \partial_x I(u^\epsilon) \mid^2 dx \geqslant k \}
\end{eqnarray*}
then using \eqref{eq:regine} and Fatou's lemma we have
\begin{align*}
c & \geqslant \liminf_{\epsilon \rightarrow 0}\int_0^T\int_\Omega f_\epsilon(u^\epsilon) \mid \partial_x I(u^\epsilon) \mid^2 dx dt \\
& \geqslant \liminf_{\epsilon \rightarrow 0}\int_{A_k}\int_\Omega f_\epsilon(u^\epsilon) \mid \partial_x I(u^\epsilon) \mid^2 dx dt \\
& \geqslant \int_{A_k}\liminf_{\epsilon \rightarrow 0}\int_\Omega f_\epsilon(u^\epsilon) \mid \partial_x I(u^\epsilon) \mid^2 dx dt \\
& \geqslant k \mid A_k \mid.
\end{align*}
So $ \mid A_k \mid \leqslant \frac{c}{k} $ and the set
\begin{eqnarray*}
\left\{ t \in P ; \liminf_{\epsilon \rightarrow 0}\int_\Omega f_\epsilon(u^\epsilon) \mid \partial_x I(u^\epsilon) \mid^2 dx = \infty \right\}
\end{eqnarray*}
has measure zero. We deduce that for all $ t \in P $
\begin{eqnarray*}
\liminf_{\epsilon \rightarrow 0}\int_\Omega \mid \partial_x I(u^\epsilon) \mid^2 dx < \infty
\end{eqnarray*}
and so for all $ t \in P $
\begin{eqnarray*}
u^\epsilon(t,.) \rightharpoonup u(t,.) \text{ in } H_N^{\alpha + 1}(\Omega)- \text{weakly.}
\end{eqnarray*}
In particular, we can pass to the limit in the flux $ J_\epsilon = f_\epsilon(u^\epsilon) \partial_x I(u^\epsilon) $ and write
\begin{eqnarray*}
\lim_{\epsilon \rightarrow 0} J_\epsilon = J = f(u) \partial_x I(u) \text{ in } L^1(\Omega) \text{ and for almost } t \in (0,T).
\end{eqnarray*}
Finally, since $ u \in H_N^{\alpha +1}(\Omega) $, $ u_x(t,x) = 0 $ for $ x \in \partial\Omega $ and almost every $ t \in (0,T) $.

\paragraph{Case $ 1 \leqslant \alpha < 2.$}

We recall that $ (\partial_t u^\epsilon)$ is bounded in $ L^2(0,T; W^{-1,2}(\Omega)) $. So Aubin's lemma implies that $ (u^\epsilon) $ converges in $ L^2(0,T; C^{0,\beta}(\Omega)) $ for all $ \beta < 1 $. We can thus find a subsequence, also denoted $ (u^\epsilon) $, and a set $ P \subset (0,T) $ such that $ \mid (0,T) \setminus P \mid = 0 $ and for all $ t \in P $ , $ u^\epsilon(t,.) $ converges strongly in $ C^\beta(\Omega) $.

We note that for all $ t \in P $, $ u $ is strictly positive. Indeed if there exists $ (t_0,x_0) \in P \times \Omega  $ such that $ u(t_0,x_0) = 0 $ then for any $ \beta < 1 $ there exists a constant $ c_\beta $ such that for all $ x \in \Omega $
\begin{eqnarray*}
u(t_0,x) \leq c_\beta \mid x - x_0 \mid^\beta.
\end{eqnarray*}
Thus
\begin{eqnarray*}
\int G(u(x,t_0)) dx \geq \int \frac{1}{(c_\beta \mid x - x_0 \mid^\beta)^{n-2}} dx.
\end{eqnarray*}
Given $ n > 3,$ we can choose $ \beta < 1 $ such that $ \beta (n-2) > 1 $. We deduce
\begin{eqnarray*}
\int G(u(x,t_0)) dx = \infty
\end{eqnarray*}
which contradicts \eqref{eq: contra}.

The rest of the proof is the same as in the first case.

\appendix

\section{Proof of Proposition \ref{prop:holder}}

Our aim is to prove that if
\begin{eqnarray} \label{eq:K}
\mid u_\delta(t,x_2) - u_\delta(t,x_1) \mid \leq K \mid x_2 - x_1 \mid^\gamma
\end{eqnarray}
for all $t \in (0,T), x_1 $ and $x_2 \in \Omega$ with constant $ K$ independent of $\delta$ and $T$, then there exists a constant $M$ independent of $\delta$ and $T$ such that
\begin{eqnarray} \label{eq: M}
\mid u_\delta(t_2,x_0) - u_\delta(t_1,x_0) \mid \leqslant M \mid t_2 - t_1 \mid^\frac{\gamma}{2\gamma +3}
\end{eqnarray}
for all $t_1$ and $t_2 \in (0,T), x \in \Omega.$
This proof is an adaptation of the proof done by Bernis-Friedman in case $\gamma = \frac{1}{2}$ \cite[Lemma 2.1]{bernis} for a general $\gamma$.

We suppose that for all $M>0$ one can find $x_0 \in \Omega$ and $t_2, t_1 \in (0,T)$ such that
\begin{eqnarray}
\mid u_\delta(t_2,x_0) - u_\delta(t_1,x_0) \mid > M \mid t_2 - t_1 \mid^\frac{\gamma}{2\gamma +3}.
\end{eqnarray}
We suppose that $u_\delta(t_2,x_0) > u_\delta(t_1,x_0)$ and that $t_2 > t_1$; thus
\begin{eqnarray}
u_\delta(t_2,x_0) - u_\delta(t_1,x_0) > M (t_2 - t_1)^\mu, \quad \quad 0<t_1<t_2<T,
\end{eqnarray}
where $\mu = \frac{\gamma}{2\gamma +3}$.
We have
\begin{eqnarray}
\iint u_\delta \partial_t \varphi = - \iint h_\delta \partial_x \varphi
\end{eqnarray}
where $h_\delta = u_\delta^n \partial_x I(u_\delta)$, which is valid for any reasonable testfunction. Consider a testfunction $\varphi$ of the form
\begin{eqnarray*}
\varphi(t,x) = \xi(x) \theta_\rho(t)
\end{eqnarray*} 
where $\xi$ and $\theta_\rho$ are defined as follows.
\begin{eqnarray*}
\xi(x) = \xi_0\left(\frac{x-x_0}{\left(M/4K\right)^\frac{1}{\gamma}(t_2 - t_1)^\frac{\mu}{\gamma}}\right)
\end{eqnarray*}
where $M$ is from \eqref{eq: M} and $K$ is from \eqref{eq:K}, and $\xi_0(x) = \xi_0(-x)$, $\xi_0 \in C^\infty_0(\Omega)$, $\xi_0(x) = 1$ if $0\leqslant x < \frac{1}{2}$, $\xi_0(x) = 0$ if $x \geqslant 1$ and $\xi_0'(x) \leqslant 0$ if $x \geqslant 0$. Thus 
\begin{eqnarray*}
\xi(x) = \begin{cases} 0 \quad \quad \text{if } \mid x - x_0\mid \geqslant \left(M/4K\right)^\frac{1}{\gamma}(t_2 - t_1)^\frac{\mu}{\gamma} \\
1 \quad \quad \text{if } \mid x - x_0\mid \leqslant \frac{1}{2} \left(M/4K\right)^\frac{1}{\gamma}(t_2 - t_1)^\frac{\mu}{\gamma}.
\end{cases}
\end{eqnarray*}
We take
\begin{align*}
\theta_\rho(t) = \int_{-\infty}^t \theta_\rho'(s) ds \quad \text{where } \quad \theta_\rho'(t) = \begin{cases} \frac{1}{\rho} \quad & \text{if } \mid t - t_2 \mid < \rho \\
\frac{-1}{\rho} \quad & \text{if } \mid t - t_1 \mid < \rho \\
0 \quad & \text{elsewhere},
\end{cases}
\end{align*}
and $\rho < \frac{1}{2} (t_2 - t_1)$. 
So, we get
\begin{eqnarray*}
\iint u_\delta \xi(x) \theta_\rho'(t) = - \iint h_\delta \xi'(x) \theta_\rho(t).
\end{eqnarray*}
The left-hand side satisfies 
\begin{eqnarray*}
\iint u_\delta(t,x) \xi(x) \theta_\rho'(t) \rightarrow 4\int \xi(x) (u_\delta(t_2,x) - u_\delta(t_1,x)) dx \quad \text{ as } \rho\rightarrow 0
\end{eqnarray*}
To estimate the last expression, we shall only consider values of $x$ such that 
\begin{eqnarray*}
 \mid x - x_0 \mid \leqslant \left(M/4K\right)^\frac{1}{\gamma}(t_2 - t_1)^\frac{\mu}{\gamma}.
\end{eqnarray*}
For such values,
\begin{align*}
u_\delta(t_2,x) - u_\delta(t_1,x) & = [u_\delta(t_2,x) - u_\delta(t_2,x_0)] + [u_\delta(t_2,x_0) - u_\delta(t_1,x_0)] + [u_\delta(t_1,x_0) - u_\delta(t_1,x)] \\
& \geqslant -2K\mid x - x_0 \mid^\gamma + M (t_2 - t_1)^\mu \\
& \geqslant \frac{M}{2} (t_2 - t_1)^\mu.
\end{align*}
Hence, by assuming that the set $\{\xi = 1\}$ is included in $\Omega$ and by a change of variables in $x$,
\begin{eqnarray*}
\int \xi(x) (u_\delta(t_2,x) - u_\delta(t_1,x)) dx \geqslant \left(\int \xi_0(x) dx\right) \frac{M}{2} (t_2 - t_1)^\mu \frac{M^\frac{1}{\gamma}}{(4K)^\frac{1}{\gamma}}(t_2 - t_1)^\frac{\mu}{\gamma}.
\end{eqnarray*}
On the other hand, we have
\begin{eqnarray*}
\left| \iint h_\delta \xi'(x) \theta_\rho(t) \right| \leqslant \left( \iint h_\delta^2 \right)^\frac{1}{2} \left( \iint (\xi' \theta_\rho)^2 \right)^\frac{1}{2}.
\end{eqnarray*}
But $\xi'(x) =\left( \left(M/4K\right)^\frac{1}{\gamma}(t_2 - t_1)^\frac{\mu}{\gamma}\right)^{-1} \xi_0'\left(\frac{x-x_0}{\left(M/4K\right)^\frac{1}{\gamma}(t_2 - t_1)^\frac{\mu}{\gamma}}\right)$, so since $h_\delta$ is uniformly bounded in $L^2(Q)$ we have
\begin{eqnarray*}
\left| \iint h_\delta \xi'(x) \theta_\rho(t) \right| \leqslant \frac{C}{\frac{M^\frac{1}{\gamma}}{(4K)^\frac{1}{\gamma}}(t_2 - t_1)^\frac{\mu}{\gamma}} \left( \iint h_\delta^2 \right)^\frac{1}{2} \frac{M^\frac{1}{2 \gamma}}{(4K)^\frac{1}{2 \gamma}}(t_2 - t_1)^\frac{\mu}{2 \gamma} (t_2 - t_1 - 2\rho)^\frac{1}{2}.
\end{eqnarray*}
Thus by letting $\rho\rightarrow 0$ we conclude that
\begin{eqnarray*}
M^{1+\frac{1}{\gamma}} (t_2 - t_1)^{\mu + \frac{\mu}{\gamma}} \leqslant C M^{-\frac{1}{2\gamma}} (t_2 - t_1)^{\frac{\mu}{2 \gamma} - \frac{\mu}{\gamma} + \frac{1}{2}},
\end{eqnarray*}
where $C$ is a new constant independent of $\delta$, $T$ and $M$, thus
\begin{eqnarray*}
M \leqslant c^\frac{2\gamma}{3+2\gamma} (t_2 - t_1)^{-\mu +\frac{\gamma}{2\gamma +3}}.
\end{eqnarray*}
Since $\mu = \frac{\gamma}{2\gamma +3}$, we find that $M \leqslant C^\frac{2\gamma}{3+2\gamma}$, and the lemma follows.

\section*{Acknowledgments}

The author would like to thank C. Imbert and F. Vigneron for their expert advice and encouragement throughout working on this research paper.

\bibliographystyle{siam}
\bibliography{mabiblio}

\end{document}